\documentclass[bj,preprint]{imsart}
\setattribute{journal}{name}{}
\usepackage{latexsym,epsfig,amssymb,amsmath,amsthm,url,bbm,enumerate,float}
\usepackage[usenames,dvipsnames]{color}
\usepackage{graphicx,hyperref}
\RequirePackage{natbib}
\allowdisplaybreaks 
\numberwithin{equation}{section}
\newtheorem{Theorem}{Theorem}[section]
\newtheorem{Lemma}[Theorem]{Lemma}
\newtheorem{Proposition}[Theorem]{Proposition}

\newtheorem{Assumption}[Theorem]{Assumption}
\newtheorem{Definition}[Theorem]{Definition}

\theoremstyle{definition}
\newtheorem{Remark}[Theorem]{Remark}
\theoremstyle{definition}

\def\cinP{\stackrel{P}{\to}}

\DeclareMathOperator{\sgn}{sgn}

\def\E{\mathbb{E}}
\def\bM{\mathbb{M}}

\def\P{\mathbb{P}}
\def\R{\mathbb{R}}

\def\C{\mathbb{C}}

\def\D{\mathbb{D}}

\definecolor{darkred}{RGB}{139,0,0}
\definecolor{darkgreen}{RGB}{0,139,0}

\begin{document}

\begin{frontmatter}

\title{{Detecting tail behavior: \\ mean excess plots with confidence bounds}}
\runtitle{Mean excess plots}
\begin{aug}
  \author{\fnms{Bikramjit}  \snm{Das}\ead[label=e1]{bikram@sutd.edu.sg}}
 \and
  \author{\fnms{Souvik}  \snm{Ghosh}\ead[label=e2]{sghosh@linkedin.com}}

 \thankstext{T1}{Bikramjit Das gratefully acknowledges Support from SRG ESD 047 and MIT-SUTD IDC grant IDG31300110.}


  \runauthor{B. Das \and S. Ghosh}

  \affiliation{SUTD\thanksmark{m1} \and LinkedIn\thanksmark{m2}}

  \address{Singapore University of Technology and Design\\20 Dover Drive, Singapore 138682 \\
           \printead{e1}}

  \address{LinkedIn Corporation\\2029 Stierlin Court, Mountain View, CA 94043, USA\\
          \printead{e2}}

\end{aug}

\begin{abstract}
In many practical situations exploratory plots are helpful in understanding tail behavior of sample data. The Mean Excess plot is often applied in practice to understand the right tail behavior of a data set. It is known that if the underlying distribution of a data sample is in the domain of attraction of a Fr\'echet, Gumbel or Weibull distributions then the ME plot of the data tend to a straight line in an appropriate sense, with positive, zero or negative slopes respectively. In this paper we construct confidence intervals around the ME plots which assist us in ascertaining which particular maximum domain of attraction the data set comes from. We recall weak limit results for the Fr\'echet domain of attraction, already obtained in \cite{das:ghosh:2013} and derive weak limits for the Gumbel and Weibull domains in order to construct confidence bounds. We test our methods on both simulated and real data sets. 
\end{abstract}



\end{frontmatter}

\section{Introduction}\label{sec:intro}

  This article concerns the use of mean excess plots, a popular exploratory tool used to understand the tail behavior of a univariate data set.  Given a sample of data points, one of the first things a sensible data analyst does is to compute a summary statistics. Such a summary statistic might involve calculating measures of central tendencies (mean, median, mode) and those of dispersion (standard deviation, range, etc), plotting a histogram, an empirical cumulative distribution function and so on and so forth. A more curious analyst would ask the question, does it even make sense to calculate the sample mean or standard deviation; would they represent their counterparts in the original population?  What if the probability distribution of the population from which the data is sampled does not even have a first or second moment. This is a question that would or perhaps should particularly come to the mind of analysts modeling risk or other extreme events. In a world where data is being used to make serious economic, financial or environmental policy decisions, understanding extreme risks, which relate to the tail behavior of data sets have become increasingly important. This can be easily observed in the world of finance and insurance \citep{das:embrechts:fasen:2013,Donnelly:Embrechts,mcneil:frey:embrechts:2005}, telecommunications \citep{maulik:resnick:rootzen:2002}, environmental statistics \citep{davison:smith:1990} and many more areas.

    The  mean excess(ME) plot is a graphical tool that is widely used  to understand the tail behavior of a sample; see \cite{embrechts:kluppelberg:mikosch:1997,davison:smith:1990}. A Mean Excess plot, if the mean exists, assists in distinguishing light-tailed data sets from heavy-tailed ones. The inference is based on a visual examination of the slope of a fitted line through the ME plot (to be described in the next section) being zero, less than zero or greater than zero. Clearly, a confidence set around the fitted line would make inference in these cases more meaningful; hence this is the aim of the paper.
    
\subsection{The ME plot}\label{subsec:meplot}

The ME plot, as described in the introduction, is a popular tool in extreme value analysis. It is a simple graphical test to check whether data conform to a generalised Pareto distribution (GPD) . The class of GPD arise naturally in extreme value analysis as limit distributions while using the peaks-over-threshold (POT) method \citep{beirlant:goegebeur:teugels:segers:2004, davison:smith:1990}. The cumulative distribution function of a GPD is:

\begin{equation}\label{eq:GPD}
	G_{ \xi ,\beta }(x)= \left\{   \begin{array}{ll}  1-(1+  \xi x/ \beta  )^{ -1/\xi }  &\mbox{ if }\xi \neq 0, \\
        1-\exp(-x/\beta ) & \mbox{ if } \xi =0, \end{array}  \right.
\end{equation}
where $ \beta >0$, and $ x\ge 0$ when $ \xi \ge 0$ and $ 0\le x\le -\beta/\xi $ if $ \xi <0$. Parameters $ \xi $ and $ \beta $ are referred to as the \emph{shape} and the \emph{scale} parameter respectively. In extreme value analysis, we are interested in the shape parameter $\xi$ which tells us whether the data is heavy-tailed ($\xi>0$) or light-tailed ($\xi\le0$) or even more specifically if the underlying distribution has  finite right end-point ($\xi<0$).
The case $\xi>0$ and $\beta=1$ corresponds to the classical Pareto law with tail exponent $1/\xi$.

The ME plot is an empirical graphical plot of the \emph{ME function} of a random variable $X\sim F$ which is defined as:

\begin{equation}\label{eq:meanexcess}
	M(u):= \E \big[  X-u|X>u \big],
\end{equation}
provided $\E X_+<\infty$. The ME function is also  known as the \emph{mean residual life function} for non-negative random variables and is extensively used in reliability theory and survival analysis for data modelling since $M(u)$ completely determines $F$ if  $\E (X) < \infty$ \citep{hall:wellner:1981}. Suppose we have an iid sample $X_1, \ldots, X_n \sim F$. A natural estimate of  $M(u)$  is the empirical ME function
$ \hat M(u)$ defined as
\begin{equation}\label{eq:empiricalmeanexcess}
	\hat M(u)= \frac{ \sum_{ i=1}^{ n}(X_{ i}-u)I_{ [ X_{ i}>u]}}{ \sum_{ i=1}^{ n}I_{ [ X_{ i}>u]}}, \ \ \ \ u\ge 0.
\end{equation}
Denoting $X_{(1)}\ge X_{(1)} \ge \ldots \ge X_{(n)}$ to be the order statistics from a sample $X_1,\ldots, X_n$, the ME plot is a plot of the points
 $$ \mathcal{ME}_n := \{   (X_{ (k)},\hat M(X_{ (k)})):1< k\le n  \}.$$

We study the asymptotic properties of $\mathcal{ME}_n$ for different classes of distributions $F$. It is well-known that for a random variable $ X\sim G_{ \xi ,\beta }$ , we have $E(X)<\infty$ if and only if $ \xi <1$ and in this case, the ME function of $X$ is linear in $u$:
\begin{equation}\label{eq:megdp}
	M(u)= \frac{  \beta}{1-\xi  }+ \frac{ \xi}{1-\xi}u,
\end{equation}
where $ 0\le u< \infty $ if $ 0\le \xi<1$ and $ 0\le u\le -\beta /\xi$ if $ \xi<0$.  

Interestingly, the linearity of the ME function characterises the GPD class \citep{mcneil:frey:embrechts:2005,embrechts:kluppelberg:mikosch:1997}. From the discussions in \cite{ghosh:resnick:2009}
we have learnt that the empirical ME plot $\mathcal{ME}_n$ above a high order statistics $X_{(k)}$ when appropriately normalised converge in probability to a straight line if $F$ is in the maximal domain of attraction of any generalized extreme value distribution with finite mean (Gumbel, Weibull or Fr\'echet distribution).
Distributional limits for $\mathcal{ME}_n$ in a space of closed sets and confidence intervals around $\mathcal{ME}_n$ can also be computed in many cases and such findings have been discussed in case the underlying data is heavy-tailed (Fr\'echet domain of attraction) in \cite{das:ghosh:2013}.

In Section \ref{subsec:misc} we collect notations and ideas to be used throughout the paper. See \cite{das:ghosh:2013} for further elaboration of the concepts of convergence of closed sets (random) in this context. In the main part of the paper we start by consolidating a few results which are already known on the distributional property of ME plot, especially in the heavy-tailed case \citep{das:ghosh:2013}; this is covered in Section \ref{subsec:MEplot:known}. The rest of Section \ref{sec:MEplot} deals with limit results for ME plots in the case where the underlying distribution is either in the Gumbel maximum domain of attraction or in a Weibull maximum domain of attraction. The limit theorems proved is Section \ref{sec:MEplot} is used to create confidence bounds around the ME plots in Section \ref{sec:CI}. In Section \ref{sec:data} we use the tools developed in Section \ref{sec:MEplot} and \ref{sec:CI} to test it out both on simulated data as well as real data sets.

\subsection{Miscellany}\label{subsec:misc}
First we recall the idea of maximum domain of attraction of an extreme value distribution. The class of extreme value distributions is parametrized by a shape parameter $\xi\in \R$ and we define the distribution function $G_{\xi}$ to be
 \[G_{\xi}(x) = \exp(-(1+\xi x)^{-1/\xi}), \quad 1+\xi x >0,\]
for all real $\xi$ and for $\xi=0$, the right hand side is interpreted as $\exp(-e^{-x})$. 
\begin{Definition}\label{def:extremevaluedist}
A distribution function $F$ (or the underlying random variable $X\sim F$) is in the \emph{maximum domain of attraction} of an extreme value distribution $G_{\xi}$ if there exists sequences $c_n>0$ and $d_n\in \R$ such that 
\[ F^n(c_nx + d_n) \to G_{\xi}(x) \quad \text{for all } x\in \Re.\]
\end{Definition}
The distributions for the cases $\xi>0$, $\xi=0$ and $\xi<0$ are respectively called the Fr\'{e}chet distribution, the Gumbel distribution and the Weibull distribution. As mentioned in the introduction, if $F\in D(G_{\xi})$ for some extreme value distribution  with $\xi<1$, implying that $F$ has finite mean, then the ME function of  $F$ is linear with an appropriate slope determined by the parameter $\xi$; see \cite{ghosh:resnick:2009}.

 Throughout this paper we will take $ k:=k_{ n}$ to be a sequence
increasing to infinity such that  $n/k_n \to \infty$ or $ k_{ n}/n\to 0$.  For a distribution
function $F(x)$ we write $\bar F(x):=1-F(x) $ for the tail and the
quantile function is
$$b(u):=F^\leftarrow (1-\frac 1u):=\inf\left\{s:F(s)\geq
1-\frac 1u\right\}=\Bigl(\frac{1}{1-F}\Bigr)^\leftarrow (u).$$
  A function
$U:(0,\infty)\to \mathbb{R}_+$ is regularly varying with index
$\rho \in \mathbb{R}$, written $U\in RV_\rho$, if
$$\lim_{t\to\infty} \frac{U(tx)}{U(t)}=x^\rho,\quad x>0.$$
 If $X\sim F$ we will often have the right-hand tail of $F$ to be regularly varying, that is, $\bar{F}\in RV_{-\alpha}$ for $\alpha\ge 0$, and by abuse of notation we might say $X\in RV_{-\alpha}$. Regular variation is discussed in several books
such as \cite{resnickbook:2007, resnickbook:2008, seneta:1976,
geluk:dehaan:1987, dehaan:1970, dehaan:ferreira:2006, bingham:goldie:teugels:1987}.

We use $\bM_{+}(0,\infty]$ to denote  the space of nonnegative Radon measures $ \mu$ on $
(0,\infty]$ metrized by the vague metric.  Point measures are written as a function of their points $\{x_i,
i=1,\dots,n\}$ by $\sum_{i=1}^n \delta_{x_i}.$ See, for example,
\cite[Chapter 3]{resnickbook:2008}.

 We will use the following notations to denote different classes of functions: For $  0\le a<b\le \infty$,
\begin{enumerate}
	\item $ \C[a,b)$: Continuous functions on $ [a,b)$.
        \item  $ \D[a,b)$: Right-continuous functions with finite left limits defined on $ [a,b)$.
        \item $ \D_{ l}[a,b)$: Left-continuous functions with finite right limits defined on $ [a,b)$.
\end{enumerate}
It is known that $\D[0,1]$ is complete and separable under a metric $d_0(\cdot)$ which is equivalent to the Skorohod metric $d_S(\cdot)$ \citep[p.128]{billingsley:1968},
 but not under the uniform metric $ \|\cdot \|$. As we will see,  the limit processes that appear in our analysis below are always continuous. We can
 check that if $x$ is continuous (in fact uniformly continuous) in $[0,1]$, for $x_n \in \D[0,1]$, $||x_n-x|| \to 0$ is equivalent to $d_S(x_n,x) \to 0$
 and hence equivalent to $d_0(x_n,x) \to 0$ as $n \to \infty$ \citep[p.124]{billingsley:1968}. So we use convergence in uniform metric, for our convenience
 henceforth. For spaces of the form $ \D[a,b)$ or $ \D_{ l}[a,b)$  we will consider the topology of local uniform convergence. In some cases we will also consider
 product spaces of functions and then the topology will be the product topology. For example, $ \D^{ 2}_{ l}[1,\infty)$ will denote the class of 2-dimensional functions on $ [1,\infty)$ which are left-continuous with right limit.  The classes of functions defined on the sets $ [a,b]$ or $ (a,b]$ will have the obvious notation. For further details on notions of convergence and topology for convergences of plots see \cite{das:ghosh:2013}.

\section{Limit results for the ME Plots}\label{sec:MEplot}
In this section we find distributional limits for ME plots when it exists. We continue the study of ME plots from \cite{ghosh:resnick:2009} and \cite{das:ghosh:2013} and give a complete picture of limit results for ME plots. We cite some of the results from the afore-mentioned papers for completeness. 

The basic assumption is that we have an iid sample of data points from some unknown distribution $F$ which belongs to the maximum domain of attraction of one of the three extreme value distributions. The assumption of independence in the sample can be relaxed a bit under certain conditions which we do not explore here.

Suppose $ X_{ 1},\ldots,X_{ n}$ is an i.i.d. sample from a distribution $ F$. We will work under this assumption for the entire section. The properties of the empirical ME function
 $ \hat M(u)$  as an estimator of $ M(u)$ has been studied by  \cite{yang:1978}.  It was shown there that $ \hat M(u)$ is uniformly strongly consistent for
$ M(u)$: for any $ 0<b<\infty$
\[
 	P\left[ \lim_{ n \rightarrow \infty } \sup_{ 0\le u\le b} \left| \hat M(u)-M(u) \right|= 0 \right]=1.
\]
A weak (distributional) limit for $ \hat M(u)$ was also shown in \cite{yang:1978}: for any $ 0<b<1$
\[
 	   \sqrt{n} \left( \hat M\big(F^{ \leftarrow }(t)\big) -M\big(F^{ \leftarrow }(t)\big) \right) \Rightarrow U(t),
\]
where $U(t) $ is a Gaussian process on $ [0,b]$ with covariance function
\[
 	\Gamma (s,t)= \frac{ (1-t)\sigma ^{ 2}(t)-t\theta ^{ 2}(t)}{(1-s)(1-t)^{ 2} }    \ \ \ \ \mbox{ for all }0\le s\le t\le b
\]
with
\[
 	\sigma ^{ 2}(t)=var \left( XI_{ [t < F(X)\le 1]} \right)  \ \  \mbox{ and }  \ \ \theta(t)=E \left( XI_{ [t < F(X)\le 1]} \right) .
\]
Using Lemma 2.4 in \cite{das:ghosh:2013} it is easy to see   that  the ME plots also exhibit the same features. Our interest in ME plots is for detecting right tail behavior of data samples (an equivalent case can be easily made for left tail behavior). Hence the linearity we seek in the ME plot will be for high thresholds. Necessarily, the ME plots we will discuss in the various cases will be transformations of the ME plot over an appropriate quantile, i.e., $\{(X_{(i)},\hat{M}(X_{(i)})): 1<i\le k\}$ for $k:=k(n)<n$ where $\hat{M}$ is as defined in \eqref{eq:empiricalmeanexcess}.

\subsection{ME plot in the Fr\'echet case}\label{subsec:MEplot:known}
First we look at the case where the underlying distribution $F$ is heavy-tailed, in the sense that $F\in D(G_{\xi})$ with $\xi>0$ or in other words, $\bar{F}\in RV_{-1/\xi}$. We define the ME plot as:
\begin{align}\label{def:ME_frechet}
\mathcal{M}_n:= \frac{1}{X_{(k)}}\left\{\left(X_{(i)},\hat{M}(X_{(i)})\right): i=2, \ldots,k\right\}.
\end{align}
From \cite[Theorem 3.2]{ghosh:resnick:2009}, we know that for $0<\xi<1$, as $ n,k,n/k\to \infty $,
\[\mathcal{M}_n \cinP \mathcal{M}:= \left\{ \left(t,\frac{\xi}{1-\xi}t\right): t>1\right\}\]

The distributional behavior of $\mathcal{M}_n$ depends on whether $F$ has finite second moment or not and has been discussed under certain regularity conditions in \cite{das:ghosh:2013}. We note them down below.

\noindent
{\bf{Case 1 ($0<\xi<1/2$):}} For any $ 0<\epsilon <1 $ as $ n,k,n/k\to \infty $,
\begin{align}
 	&\mathcal{MN}_{ n}:=\left\{\Bigg( \Bigg(\frac{ i}{k }\Bigg)^{ -\xi}, \frac{ \xi}{1-\xi } \Bigg( \frac{ i}{k } \Bigg)^{ -\xi} \Bigg)  \right. \nonumber\\
	& \ \ \ \ \ \ \ \ \ \ \ \ +\left. \sqrt{k}\left( \frac{ X_{ (i)}}{X_{ (k)}  }  -\Big(\frac{ i}{k }\Big)^{ -\xi}, \frac{\hat M(X_{ (i)})}{X_{ (k)}} - \frac{ \xi}{1-\xi } \Big( \frac{ i}{k } \Big)^{ -\xi}  \right) :i=\lceil \epsilon k\rceil ,\ldots,k\right\} \nonumber\\
	& \Rightarrow \mathcal{MN}:= \left\{\left( t^{ -\xi}+\xi t^{ -(1+\xi)}B(t),\frac{ \xi}{1-\xi }t^{ -\xi}+  \xi t^{ -1} \int_{ 0}^{ t} y^{ -(1+\xi)}B(y)dy \right) , \epsilon \le t\le 1\right\} \ \ \ \ \mbox{ in } \mathcal{F},\label{eq:frechetwithvar}
\end{align}
where $ B(t)$ is the standard Brownian bridge on $ [0,1]$. This is the case where $F$ has a finite second moment and hence the distributional limit has a Brownian component.\\
\noindent
{\bf{Case 2 ($ 1/2<\xi<1$):}} For any $ 0<\epsilon <1 $, as $ n,k,n/k\to \infty $,
 \begin{align}
 	\mathcal{MN}_{ n}:= &\Bigg\{\Bigg( \Bigg(\frac{ i}{k }\Bigg)^{ -\xi}, \frac{ \xi}{1-\xi } \Bigg( \frac{ i}{k } \Bigg)^{ -\xi} \Bigg) \nonumber\\ &+ \left(\sqrt{k}\left( \frac{ X_{ (i)}}{X_{ (k)}  }  -\Big(\frac{ i}{k }\Big)^{ -\xi}\right), \frac{ kb(n/k)}{b(n) } \left(\frac{\hat M(X_{ (i)})}{X_{ (k)}} - \frac{ \xi}{1-\xi } \Big( \frac{ i}{k } \Big)^{ -\xi} \right) \right) :i=\lceil \epsilon k \rceil,\ldots,k\Bigg\} \nonumber\\
	 \Rightarrow & \mathcal{MN}:= \left\{\left( t^{ -\xi}+\xi t^{ -(1+\xi)}B(t),\frac{ \xi}{1-\xi }t^{ -\xi}+  t^{ -1} S_{ 1/\xi } \right) , \epsilon \le t\le 1\right\} \ \ \ \ \mbox{ in } \mathcal{F}, \label{eq:frechetwithoutvar}
\end{align}
where $ S_{ 1/\xi }$ is a stable random variable  with characteristic function 
\begin{equation}\label{eq:charstable}
 	E\big[  e^{ itS_{ 1/\xi}} \big]   =
	\exp \Bigg\{ -\frac{ 1}{1-\xi }\Gamma \Big( 2- \frac{ 1}{\xi } \Big)\cos  \frac{ \pi}{2\xi }| t |^{ 1/\xi}\Big[  1-i\sgn(t) \tan \frac{ \pi}{2\xi } \Big]   \Bigg\},
\end{equation}
and is independent of the standard Brownian Bridge $ B(t)$ on $[0,1]$.
	This is the case where $F$ has a finite mean but does not have a finite second moment, hence we also observe a non-Gaussian stable weak limit. The results in \eqref{eq:frechetwithvar} and \eqref{eq:frechetwithoutvar} are described in Theorems 4.3 and 4.6 in details in \cite{das:ghosh:2013}.

\subsection{ME plot in the Gumbel case}\label{subsec:MEplot:gumbel}

The behaviour (in probability) of ME plot when $F$ is in the maximum domain of attraction of a Gumbel distribution has been discussed in \cite{ghosh:resnick:2009}. We state the following result to recall notations to be used: this follows from Theorems 3.3.26 and 3.4.13(b) in \cite{embrechts:kluppelberg:mikosch:1997}; see  \citep[Theorem 3.9]{ghosh:resnick:2009} or \citep[Proposition 1.4]{resnickbook:2008} for further details.

\begin{Proposition}\label{prop:char:zero:xi}
The following are equivalent for a distribution function $ F$ with right end point $ x_{ F}\le \infty$:
 \begin{enumerate}
   \item $ F$ is in the maximum domain of attraction of the Gumbel distribution, i.e.,
\begin{equation}\label{eq:Gumbel:defn}
 	F^{ n}\big( c( n )x +d( n )\big)\to \exp \big\{-e^{ -x}  \big\}   \ \ \ \ \mbox{ for all }  x\in \mathbb{R},
\end{equation}
for some sequence $ c(n) $ and $ d(n) $. 

        \item  There exists  $ z<x_{ F}$ such that $ F$ has a representation
        \begin{equation}\label{eq:char:F:gumbel}
 	\bar F(x)=\kappa (x)\exp \Big\{ -\int_{ z}^{ x} \frac{ 1}{a(t) }dt \Big\} , \ \ \ \ \mbox{ for all } z<x<x_{ F},
\end{equation}
where $ \kappa (x)$ is a  measurable function satisfying $ \kappa (x)\to \kappa >0$,  $ x\rightarrow x_{ F}$, and $ a(x)$ is a positive, absolutely continuous function with density $ a^{ \prime}(x)\to 0$ as $ x\to x_{ F}$.

\end{enumerate}
\end{Proposition}
We know from \citep[Proposition 1.1]{resnickbook:2008} that a choice of the norming sequence $ c(n) $ and $ d(n) $ in \eqref{eq:Gumbel:defn} is
\[
	d(n)=F^{ \leftarrow }(1-n^{ -1}) \ \ \ \  \mbox{ and }\ \ \ \ c(n)=a(d(n)).
\]
Theorem 3.3.26 in \cite{embrechts:kluppelberg:mikosch:1997} says that a choice of the auxiliary function $ a(x)$ in \eqref{eq:char:F:gumbel} is
\[
 	a(x)   =\int_{ x}^{ x_{ F}} \frac{ \bar F(t)}{ \bar F(x) } dt \ \ \ \ \mbox{ for all }x<x_{ F},
\]
and for this choice, the auxiliary function is  the ME function, i.e., $ a(x)=M(x)$. Furthermore, we also know that $ a^{ \prime }(x)\to 0$ as $ x\to x_{ F}$ and this  implies that $ M(u)/u\to 0$ as $ u\to x_{ F}$. Define the ME plot in this case as
\begin{align}\label{def:ME_gumbel}
\mathcal{M}_n:= \frac{1}{X_{\lceil k/e \rceil}-X_{(k)}}\left\{\left(X_{(i)}-X_{(k)},\hat{M}(X_{(i)})\right): i=2, \ldots,k\right\}.
\end{align}
From a minor modification of \cite[Theorem 3.10]{ghosh:resnick:2009}, we know that as $ n,k,n/k\to \infty $,
$\mathcal{M}_n \cinP \mathcal{M}:= \left\{ \left(t,1\right): t>0\right\}$.
Now we will additionally put one more condition in order to get a weak limit for ME plots in the Gumbel case which is stated as follows.

%
%
%
\begin{Assumption}\label{assmp:ME1:xi0}
The distribution function $ F$ satisfies the following:
\begin{equation}
 \sqrt{k} \left(\frac nk \bar{F} \left(c(n/k)y+d(n/k)\right) - e^{-y}\right)\to 0
\end{equation}
point-wise and in $ L_{1} $-norm in $ [0,\infty) $ as $n,k,n/k \to \infty$.
\end{Assumption}
Now we can state the distributional result for ME plots when $ F $ is in the maximum domain of attraction of the Gumbel distribution.
\begin{Theorem}\label{thm:weaklimit:MEplot:xi0}
Suppose $ X_{ 1},\ldots,X_{ n}$ are i.i.d. observations from a distribution $ F $ which is in the maximum domain of attraction of the Gumbel distribution and satisfies Assumption \ref{assmp:ME1:xi0}.  Then for any $ 0<\epsilon <1$, as $ n,k,n/k\to \infty$,
\begin{align*}
 	\mathcal{MN}_{ n}&:=\left\{  \left(-\ln \left(\frac ik \right), 1 \right) + \sqrt{k}\left( \frac{ X_{ (i)}-X_{(k)}}{X_{(\lceil k/e \rceil)}-X_{ (k)}  } + \ln \left(\frac ik \right), \frac{\hat M(X_{ (i)})}{X_{(\lceil k/e \rceil)}-X_{ (k)}  } - 1  \right) :i=\lceil \epsilon k\rceil ,\ldots,k\right\} \\
	& \Rightarrow \mathcal{MN}:= \left\{\left(-\ln (t)+ eB(e^{-1})\ln (t) +\frac{B(t)}t , 1 + e B(e^{-1}) + \frac{1}{t} \int\limits_0^t\frac{B(s)}{s}ds   \right) , \epsilon \le t\le 1\right\} \ \ \ \ \mbox{ in } \mathcal{F},
\end{align*}
 where $ B(t)$ is the standard Brownian bridge on $ [0,1]$.
\end{Theorem}

\begin{proof}
The proof is along the same lines of the proof of Theorem 4.3 in \cite{das:ghosh:2013}. Denote the tail empirical measures by
\begin{align}
 \nu_n(\cdot)&:= \frac1k \sum\limits_{i=1}^n \varepsilon_{\frac{X_i-d(n/k)}{c(n/k)}}(\cdot) \mbox{ and,} \\
 \hat\nu_n(\cdot)&:= \frac1k \sum\limits_{i=1}^n \varepsilon_{\frac{X_i-X_{(k)}}{c(n/k)}}(\cdot)
\end{align}
and define for $k:=k(n)<n$ and $0<t\le 1$:
\[
 W_n(t):  = \sqrt{k}\big(\nu_n(-\ln t, \infty]  - t \big)
        = \sqrt{k}\left(\frac 1k \sum\limits_{i=1}^n \varepsilon_{\frac{X_i-d(n/k)}{c(n/k)}} (-\ln t, \infty] - t\right).\label{eq:Wn1:xiless0}
\]
We prove in Lemma \ref{lem:weakfornu_nGumbel} that $ W_{n}\Rightarrow W $ in $ D_{l}(0,1] $, where $ W $ is the standard Brownian motion in $[0,1]$. Applying Vervaat's lemma \citep[Proposition 3.3, p.59]{resnickbook:2007} to \eqref{eq:Wn1:xiless0} we get
\[
	 \sqrt{k} \left( \exp \left\{- \frac{ X_{(\lceil kt \rceil)}-d(n/k)}{c(n/k) } \right\} -t   ,   \nu_n(-\ln t, \infty]  - t  \right) \Rightarrow \left( -W(t),W(t) \right)  \ \ \ \  \mbox{ in } \mathbb{D}^{2}_{l}(0,1].
\]
Using the Functional Delta-method \citep[Theorem 3.9.4]{vandervaart:wellner:1996} we get
\begin{equation}\label{eq:deltamethodresult}
	\sqrt{k} \left(  \frac{ X_{(\lceil kt \rceil)}-d(n/k)}{c(n/k) } + \ln t   ,  \nu_n(y, \infty]  - e^{-y}   \right) \Rightarrow \left( \frac{ W(t)}{t},W(e^{-y}) \right) \ \ \ \ \mbox{ in }  \mathbb{D}_{l}(0,1]\times \D [1,\infty),
\end{equation}
and it is easy to check that
\begin{equation}
	\sqrt{k} \left(  \frac{ X_{(\lceil kt \rceil)}-X_{(k)}}{c(n/k) } + \ln t   , \hat \nu_n(y, \infty]  - e^{-y}   \right) \Rightarrow \left( \frac{ B(t)}{t},B(e^{-y}) \right) \ \ \ \ \mbox{ in }  \mathbb{D}_{l}(0,1]\times \D [1,\infty).
\end{equation}
 Then following arguments used in the proof of Theorem 4.3 in \cite{das:ghosh:2013} we get
\begin{align*}
	& \sqrt{k} \left(  \frac{ X_{(\lceil kt \rceil)}-d(n/k)}{c(n/k) } + \ln t   ,    \frac{ \hat M (X_{(\lceil kt \rceil)})}{ c(n/k) } -1\right)  \\
	& =  \sqrt{k} \left(  \frac{ X_{(\lceil kt \rceil)}-d(n/k)}{c(n/k) } + \ln t   ,    \frac{k}{(\lceil kt \rceil-1) c(n/k) }  \int^{\infty}_{\frac{ X_{(\lceil kt \rceil)}-d(n/k)}{c(n/k) }}\nu_{n} (y,\infty]dy-1\right) \\
	& \Rightarrow  \left( \frac{ B(t)}{t}, \frac{ 1}{t }\int_{-\ln t}^{\infty} B(e^{-y}) dy    \right)
	=  \left( \frac{ B(t)}{t}, \frac{ 1}{t }\int_{0}^{t} \frac{B(s)}{s} ds    \right).
\end{align*}
The proof the theorem is completed by invoking Lemma 2.4 in \cite{das:ghosh:2013}.
\end{proof}

%
%
%
%

\begin{Lemma}\label{lem:weakfornu_nGumbel}
 As $n\to \infty, k \to \infty, n/k \to \infty$,
$$W_n \Rightarrow W$$
in $D_l(0,1]$ where $W$ is a Brownian motion in $D_l(0,1]$.
\end{Lemma}

\begin{proof}
We check the conditions C1-C4 of  \cite[Theorem 2.1]{rootzen:2009}. In this part of the proof whenever we write `$\sim$' between two expressions, it means the asymptotics hold for $n,k,n/k \to \infty$. Now  following the notations used in the aforementioned paper, we set $ r_{n}=\min   \{  k^{1/4}, (n/k)^{1/2} \} $ and $ l_{n}=1 $. For any $ u,v \in \R $ let
\[
	N_{n}(u,v):=\sum_{i=1}^{r_{n}} \varepsilon_{\frac{X_i-d(n/k)}{c(n/k)}} (u,v].
\]
Then for any $ \theta<x_{T}=\infty $ (since $F$ is in a Gumbel domain of attraction, it has right end point $x_T=\infty$) with $ 0\le u,v<\theta $ we have, \[ \P[N_{n}(u,v)\neq 0] \sim r_{n}\P[u\,c(n/k)+d(n/k) <X_{1} \le v\,c(n/k)+d(n/k)] \] and
\begin{align*}
	 \E \left[ N_{n}(u,v)^{2}| N_{n}(u,v) \neq 0 \right]  &= \frac{ \E \left[ N_{n}(u,v)^{2}\right]} { \P\left[ N_{n}(u,v) \neq 0 \right]}\\
	&    \sim 1+r_{n}\P\left[u\,c(n/k)+d(n/k) <X_{1} \le v\,c(n/k)+d(n/k)\right] \\
	& \le 1+const. r_{n} k/n
 \end{align*}
 which is bounded by the choice of $ r_{n} $. Hence condition C1 holds. 
 Condition C2 holds as the random variables $ X_{i} $ are assumed to be independent. Next note that for any $ 0\le u,v<\infty $,
 \begin{align*}
 	&\frac{ 1}{ r_{n} \bar F(d(n/k))} Cov \left( \sum_{i=1}^{r_{n}}\varepsilon_{\frac{X_i-d(n/k)}{c(n/k)}} (u,\infty] ,  \sum_{i=1}^{r_{n}}\varepsilon_{\frac{X_i-d(n/k)}{c(n/k)}} (v,\infty]\right)     \\
 	&  =\frac{ 1}{ \bar F(d(n/k))} Cov \left( \varepsilon_{\frac{X_1-d(n/k)}{c(n/k)}} (u,\infty] , \varepsilon_{\frac{X_1-d(n/k)}{c(n/k)}} (v,\infty]\right) \\
	& \sim \frac{ \bar F \left( (u\vee v) c(n/k) +d(n/k) \right) } {\bar F (d(n/k))} \\&\to \exp \left( - u\vee v \right) .
  \end{align*}
  Hence C3 holds and obviously C4 holds because of the choice of $ r_{n} $. Hence, by \cite[Theorem 2.1]{rootzen:2009}
  \[
	\sqrt{k}\left(\nu_n(u, \infty]  - \E\left(\nu_n(u, \infty] \right)\right) \Rightarrow G \quad \text{in }\,\, D[0,\infty),
\]
where $ G $ is a centered Gaussian process in $[0,\infty)$ with covariance function $ \exp \left( - u\vee v \right)$ and hence a time change $u\mapsto - \ln u$ gives us that $ W_{n}\Rightarrow W $  in $D_l(0,1]$ where $ W $ is a standard Brownian Motion on $[0,1]$.

\end{proof}

\subsection{ME plot in the  Weibull case}\label{subsec:MEplot:weibull}
If $F\in D(G_{\xi,\beta})$, then we have the following characterizations for the case $\xi<0$ \citep{embrechts:kluppelberg:mikosch:1997,ghosh:resnick:2009}:
\begin{Proposition}
 If $\xi<0$ then the following are equivalent:

\begin{enumerate}
\item $F$ has a finite right end point $x_F$ and $\bar{F}(x_F-x^{-1})\in RV_{1/\xi}$.
\item $F^n(x_F+c(n)x) \to \exp\{-(-x)^{-1/\xi}\}$ for all $x \le 0$ where $c(t)=x_F-F(1-\frac 1t), t\ge 1$..
\item There exists a measurable function $\beta(u)$ such that
  \[\lim\limits_{u\to x_F} \sup\limits_{u\le x\le x_F}|F_u(x)-G_{\xi,\beta(u)}(x)|=0.\]
\end{enumerate}

\end{Proposition}

Recall from \cite{ghosh:resnick:2009}, the following result on ME plots (there is a typographical error in the statement of the result there):
\begin{Proposition}\label{thm:xi:negative}
If $ X_{ 1},\ldots,X_{ n}$ are i.i.d. observations with distribution $ F$ which has a finite right end point $x_F$ and satisfies $ 1-F(x_F-x^{-1})\in RV_{ 1/\xi }$ as $x\to \infty$,  then in $ \mathcal{F}$,
 \begin{equation}\label{eq:lim:xiless0}
 	\mathcal{M}_{ n}:=\frac{ 1}{X_{(1)}-X_{ (k)}  }  \left\{   \big(X_{  (i)}-X_{(k)},\hat M(X_{ (i)})\big) :i=2,\ldots, k\right\}   \stackrel{ P}{\longrightarrow } \mathcal{M}:=\Big\{   \Big(t, \frac{ \xi}{1-\xi } (1-t) \Big):0 \le t \le 1  \Big\} .
\end{equation}

\end{Proposition}

In this paper we obtain the weak limit of the ME plot when the null hypothesis that  $ \bar F(x_F-x^{-1})\in RV_{ 1/\xi}$ for some $ \xi<0$ holds.
In the same spirit as \cite{das:ghosh:2013} we deal with the tail empirical process. Denote by $\nu_n$:
\begin{align}
 \nu_n(\cdot):= \frac1k \sum\limits_{i=1}^n \varepsilon_{\frac{x_F-X_i}{c(n/k)}}.
\end{align}
Following Theorem 4.2 in \cite{resnickbook:2007}, we can show that
\[\nu_n \Rightarrow \nu \ \ \ \ \text{in} \ \ \bM_{+}[0,\infty)\]
where $\nu[0,x)=x^{-1/\xi}, x\ge 0$.
Now define for $k:=k(n)<n$ and $y\ge 0$:
\begin{align}
 W_n(y):& = \sqrt{k}\left(\frac 1k \sum\limits_{i=1}^n \varepsilon_{\frac{x_F-X_i}{c(n/k)}} [0,y^{-\xi}) - \frac nk \bar{F} \left(x_F-c(n/k)y^{-\xi}\right)\right)\label{eq:Wn1:xi0}\\
        & = \sqrt{k}\left(\nu_n[0,y^{-\xi})  - \E\left(\nu_n[0,y^{-\xi})\right)\right). \label{eq:Wn2:xi0}
\end{align}
The next result in the spirit of \citep[Theorem 9.1]{resnickbook:2007} and also similar to Lemma \ref{lem:weakfornu_nGumbel}.

\begin{Lemma}\label{lem:weakfornu_nWeibull}
 As $n\to \infty, k \to \infty, n/k \to \infty$,
$$W_n \Rightarrow W$$
in $D[0,\infty)$ where $W$ is a Brownian motion in $D[0,\infty)$.
\end{Lemma}
The proof follows by going through the steps of the proof of Lemma \ref{lem:weakfornu_nGumbel} or \citep[Theorem 9.1]{resnickbook:2007}. Let us also assume the following:
\begin{Assumption}\label{assmp:ME1:xiless0}
$ F$ satisfies the following
\begin{align*}
[1]& \ \   \sqrt{k} \left(\frac nk \bar{F} \left(x_F-c(n/k)y\right) - y^{ -1/\xi}\right)\to 0\ \ \ \ \text{for all} \ \ y \ge 0,\\
 [2]& \ \  	\sqrt{k} \int_{ 0}^{ 1}\Big|\frac nk \bar{F} \left(x_F-c(n/k)y\right) - y^{ -1/\xi}  \Big|   dy\to 0.
 	\end{align*}

 as $n,k,n/k \to \infty$.
\end{Assumption}

\begin{Theorem}\label{thm:weaklimit:MEplot:xiless0}
Suppose $ X_{ 1},\ldots,X_{ n}$ are i.i.d. observations from a distribution $ F$ which has a finite right end point $x_F$ and satisfies $ 1-F(x_F-x^{-1})\in RV_{ 1/\xi }, \xi <0$ as $x\to \infty$ and Assumption \ref{assmp:ME1:xiless0} holds.
 Then for any $ 0<\epsilon <1$, as $ n,k,n/k\to \infty$,
 \begin{align*}
 	&\mathcal{MN}_{ n}:=\left\{\Bigg( 1-\Bigg(\frac{ i}{k }\Bigg)^{ -\xi}, \frac{ \xi}{\xi-1 } \Bigg( \frac{ i}{k } \Bigg)^{ -\xi} \Bigg)  \right.\\
	& \ \ \ \ \ \ \ \ \ \ \ \ +\left. \sqrt{k}\left( \frac{ X_{ (i)}-X_{(k)}}{X_{(1)}-X_{ (k)}  } -\left( 1-\Big(\frac{ i}{k }\Big)^{ -\xi}\right), \frac{\hat M(X_{ (i)})}{X_{(1)}-X_{ (k)}} - \frac{ \xi}{\xi-1 } \Big( \frac{ i}{k } \Big)^{ -\xi}  \right) :i=\lceil \epsilon k\rceil ,\ldots,k\right\} \\
	& \Rightarrow \mathcal{MN}:= \left\{\left( 1-t^{ -\xi}+\xi t^{ -(1+\xi)}B(t),\frac{ \xi}{\xi-1 }t^{ -\xi}+  \xi t^{ -1} \int_{ 0}^{ t} y^{ -(1+\xi)}B(y)dy \right) , \epsilon \le t\le 1\right\} \ \ \ \ \mbox{ in } \mathcal{F},
\end{align*}
where $ B(t)$ is the standard Brownian bridge on $ [0,1]$ restricted to $ (0,1]$.
\end{Theorem}

\begin{Remark}
  This result is similar to the one obtained for $\bar{F}\in RV_{-1/\xi}$ or $\xi>0$ in Theorem 4.3 of \cite{das:ghosh:2013}; the subtle difference appears in the fact that we no longer need to
restrict the range of $\xi$ as is done there with $0<\xi<1/2$, since the integral
\[
 	\int_{ 0}^{ t} y^{ -(1+\xi)}B(y)dy \stackrel{ d}{= }   \int_{ 0}^{ t} y^{ -(1+\xi)}W(y)dy-W(1)\int_{ 0}^{ t}y^{ -\xi}  dy
\]
exists if and only if $\int_0^t s^{-2\xi} ds<\infty$ which is always true for $\xi<0$ and in turn implies that the limit $\mathcal{MN}$ exists. The truncation with $\epsilon$
with $\epsilon \le t \le 1$ is still necessary to guarantee that the limit set $\mathcal{MN}$ does not blow up for $t$ near $0$.
\end{Remark}
\begin{proof}
 The proof is omitted here as it follows using similar arguments as in the proof of \cite[Theorem 4.3]{das:ghosh:2013}. The difference occurs in the
 fact that we use the weak convergence result mentioned in Lemma \ref{lem:weakfornu_nWeibull} as our basis and apply a proper version of
Vervaat's Lemma and `converging together' arguments on this to obtain the result.
\end{proof}

\section{Creating confidence bounds from the limit results}\label{sec:CI}

In Section \ref{sec:MEplot} we obtain weak limits for ME plots for different values of $\xi\in\R$ where the underlying distribution $F \in D(G_{\xi})$. Now, depending on varying values of $\xi$ we construct the different confidence bounds following the results. We resort to Monte Carlo simulation for actually computing the limits since most of them require calculating quantiles of suprema of functionals of Brownian bridges  over a finite interval or quantiles of stable distributions.

 We need to truncate the ME plot near infinity in all the cases since the weak limits we obtain  blow up there (it relates to $t$ near $0$ in the limit of $\mathcal{MN}_n$).

\subsection{Fr\'{e}chet case:}\label{subsec:CI_frechet}
 This case has already been discussed in \cite{das:ghosh:2013} and we recall it here for the sake of completeness. Define the truncated versions of $\mathcal{M}_n$ defined in \eqref{def:ME_frechet} and its limit $\mathcal{M}$   respectively for $0<\epsilon<1$ as:
 \begin{equation}\label{eq:lim:frechet:epsilon}
 	\mathcal{M}^{ \epsilon }_{ n}:=\frac{ 1}{X_{ (k)}  }  \left\{   \big(X_{  (i)},\hat M(X_{ (i)})\big) :i=\lceil k\epsilon \rceil,\ldots, k\right\}  \ \ \ \  \mbox{ and  } \ \ \ \  \mathcal{M}^{ \epsilon }:=\Big\{   \Big(t^{-1}, \frac{ \xi}{1-\xi } t^{-1} \Big):\epsilon \le t \le 1  \Big\} .
\end{equation}
Then $ \mathcal{M}_{ n}^{ \epsilon }\stackrel{ P}{\to } \mathcal{M}^{ \epsilon } $.

\noindent{\bf{Case 1 ($0<\xi<1/2$):}} From \eqref{eq:frechetwithvar}, we have that the $(1-\alpha)100 \%$ confidence band for $\mathcal{M}^{\epsilon}$  as
  \begin{align}\label{eq:conband:frechet:var}
& \mathcal{CM}_n^{\epsilon} := \mathcal{M}_n^{\epsilon} + \left\{(x,y):x\in \left(-\frac{c_{\alpha/2,\epsilon}}{\sqrt{k}},\frac{c_{\alpha/2,\epsilon}}{\sqrt{k}}\right), y\in \left( -\frac{d_{\alpha/2,\epsilon}}{\sqrt{k}},\frac{d_{\alpha/2,\epsilon}}{\sqrt{k}} \right)  \right\},
\end{align}
where \begin{align}
c_{\alpha,\epsilon} &= \text{$(1-\alpha)$-th quantile of} \sup\limits_{\epsilon \le t \le 1}\xi t^{-(1+\xi)}B(t),\label{eq:c:frechet}\\
d_{\alpha,\epsilon}& = \text{$(1-\alpha)$-th quantile of} \sup\limits_{\epsilon \le t \le 1}\xi t^{-1} \int\limits_0^t y^{-(1+\xi)}B(y)dy.\label{eq:d:frechet}
\end{align}

Since the weak limit of properly scaled and shifted $\mathcal{M}^{ \epsilon }_n$  consists of functionals of the same Brownian Bridge in both components,
\eqref{eq:conband:frechet:var} provides an asymptotic confidence bound around $\mathcal{M}^{\epsilon}$ with $P(\mathcal{M}^{\epsilon} \subset \mathcal{CM}_n^{\epsilon}) \ge (1-\alpha)$ for large $n$. 

\noindent{\bf{Case 2 ($1/2<\xi<1$):}} From \eqref{eq:frechetwithoutvar}, we have the $(1-\alpha )100\% $ confidence band for $ \mathcal{M}^{ \epsilon }$ as
  \begin{align}\label{eq:conband:frechet:novar}
\mathcal{CM}_n^{\epsilon} &=\left\{ \left( \frac{ X_{ (\lceil kt \rceil)}}{ X_{ (k)}}, \frac{ \hat M(X_{ (\lceil kt \rceil)})}{ X_{ (k)} } \right) + \left(-\frac{c_{\alpha_1/2,\epsilon}}{\sqrt{k}},\frac{c_{\alpha_1/2,\epsilon}}{\sqrt{k}}\right)\times \left( \frac{X_{ (1)} d_{1-\alpha_2/2}}{\lceil kt \rceil X_{ (k)}},\frac{X_{ (1)}d_{\alpha_2/2}}{ \lceil kt \rceil X_{ (k)}} \right): \epsilon \le t \le 1  \right\},
\end{align}
where
\begin{align*}
d_{\alpha} &= \text{$(1-\alpha)$-th quantile of } S_{ 1/\xi} \text{ defined in \eqref{eq:charstable}}.
\end{align*}
Here $0<\alpha_1,\alpha _{ 2}<1$ are chosen such that $(1-\alpha)=(1-\alpha_1)(1-\alpha_{ 2})$. Since the random components in the first and second components in the limit of \eqref{eq:frechetwithoutvar}
are independent this gives us the right confidence interval so that $P(\mathcal{M}^{\epsilon} \subset \mathcal{CM}_n^{\epsilon} ) \ge 1-\alpha$. The above quantiles are calculated using Monte Carlo simulation methods. In real data examples $\xi$ is estimated using a Hill estimator, or any reasonable estimator for the tail index of a heavy-tailed distribution.

\subsection{Gumbel case}\label{subsec:CI_gumbel}

 This is the case where $F\in D(G_0)$. Many well-known distribution functions such as exponential, normal, log-normal distributions fall into this class. First we define the truncated versions of $\mathcal{M}_n$ defined in \eqref{def:ME_gumbel} and its limit $\mathcal{M}$   respectively for $0<\epsilon<1$ as:
 \begin{equation}\label{eq:lim:gumbel:epsilon}
 	\mathcal{M}^{ \epsilon }_{ n}:=\frac{ 1}{X_{(\lceil k/e \rceil)}-X_{ (k)}  }  \left\{   \big(X_{  (i)}-X_{(k)},\hat M(X_{ (i)})\big) :i=\lceil k\epsilon \rceil,\ldots, k\right\}  \ \ \ \  \mbox{ and  } \ \ \ \  \mathcal{M}^{ \epsilon }:=\left\{   \left(- \ln(t), 1 \right):\epsilon \le t \le 1  \right\} .
\end{equation}
Then $ \mathcal{M}_{ n}^{ \epsilon }\stackrel{ P}{\to } \mathcal{M}^{ \epsilon } $. Using Theorem \ref{thm:weaklimit:MEplot:xi0}, we have that the $(1-\alpha)100 \%$ confidence band for $\mathcal{M}^{\epsilon}$  as
  \begin{align}\label{eq:conband:gumbel}
& \mathcal{CM}_n^{\epsilon} := \mathcal{M}_n^{\epsilon} + \left\{(x,y):x\in \left(-\frac{c_{\alpha/2,\epsilon}}{\sqrt{k}},\frac{c_{\alpha/2,\epsilon}}{\sqrt{k}}\right), y\in \left( -\frac{d_{\alpha/2,\epsilon}}{\sqrt{k}},\frac{d_{\alpha/2,\epsilon}}{\sqrt{k}} \right)  \right\},
\end{align}
where \begin{align}
c_{\alpha,\epsilon} &= \text{$(1-\alpha)$-th quantile of} \sup\limits_{\epsilon \le t \le 1} \left\{ eB(e^{-1})\ln (t) + \frac{B(t)}{t} \right\}\label{eq:c:gumbel}\\
d_{\alpha,\epsilon}& = \text{$(1-\alpha)$-th quantile of} \sup\limits_{\epsilon \le t \le 1} \left\{ eB(e^{-1}) + \frac{1}{t}\int\limits_0^t \frac{B(y)}{y}dy \right\}.\label{eq:d:gumbel}
\end{align}
By the same logic, as the earlier cases, 
\eqref{eq:conband:gumbel} provides an asymptotic confidence bound around $\mathcal{M}^{\epsilon}$ with $P(\mathcal{M}^{\epsilon} \subset \mathcal{CM}_n^{\epsilon}) \ge (1-\alpha)$ for large $n$. The quantiles are obtained using Monte Carlo simulation.

\subsection{Weibull case}\label{subsec:CI_weibull}
In this case $F\in D(G_{\xi})$ with $\xi<0$. Many distributions, especially with  bounded right hand-tail falls into this category, for example Uniform, Beta, etc.  Here we define the truncated versions of $\mathcal{M}_n$ as defined in Proposition \ref{thm:xi:negative} and its limit $\mathcal{M}$   respectively for $0<\epsilon<1$ as:
 \begin{equation}\label{eq:lim:gumbel:epsilon}
 	\mathcal{M}^{ \epsilon }_{ n}:=\frac{ 1}{X_{(1)}-X_{ (k)}  }  \left\{   \big(X_{  (i)}-X_{(k)},\hat M(X_{ (i)})\big) :i=2,\ldots, k\right\}  \    \mbox{ and  }   \  \mathcal{M}^{ \epsilon }:=\left\{   \left(t, \frac{\xi}{1-\xi} (t-1) \right):\epsilon \le t \le 1  \right\} .
\end{equation}
Then $ \mathcal{M}_{ n}^{ \epsilon }\stackrel{ P}{\to } \mathcal{M}^{ \epsilon } $. Using Theorem \ref{thm:weaklimit:MEplot:xiless0}, we have that the $(1-\alpha)100 \%$ confidence band for $\mathcal{M}^{\epsilon}$  as
  \begin{align}\label{eq:conband:weibull}
& \mathcal{CM}_n^{\epsilon} := \mathcal{M}_n^{\epsilon} + \left\{(x,y):x\in \left(-\frac{c_{\alpha/2,\epsilon}}{\sqrt{k}},\frac{c_{\alpha/2,\epsilon}}{\sqrt{k}}\right), y\in \left( -\frac{d_{\alpha/2,\epsilon}}{\sqrt{k}},\frac{d_{\alpha/2,\epsilon}}{\sqrt{k}} \right)  \right\},
\end{align}
where \begin{align}
c_{\alpha,\epsilon} &= \text{$(1-\alpha)$-th quantile of} \sup\limits_{\epsilon \le t \le 1}\xi t^{-(1+\xi)}B(t),\label{eq:c:weibull}\\
d_{\alpha,\epsilon}& = \text{$(1-\alpha)$-th quantile of} \sup\limits_{\epsilon \le t \le 1}\xi t^{-1} \int\limits_0^t y^{-(1+\xi)}B(y)dy.\label{eq:d:weibull}
\end{align}
The bounds obtained here are very similar to the one in the Fr\'echet case.
And using the same argument,
\eqref{eq:conband:weibull} provides an asymptotic confidence bound around $\mathcal{M}^{\epsilon}$ with $P(\mathcal{M}^{\epsilon} \subset \mathcal{CM}_n^{\epsilon}) \ge (1-\alpha)$ for large $n$. Similar to the previous cases, the quantiles are obtained using Monte Carlo simulation.

\section{Examples: Simulated and real data}\label{sec:data}

This section is devoted to application of the methodology developed for constructing confidence intervals around ME plots
as derived in the Section \ref{sec:CI}. 

 \begin{figure}[H]
 \includegraphics[width=12cm]{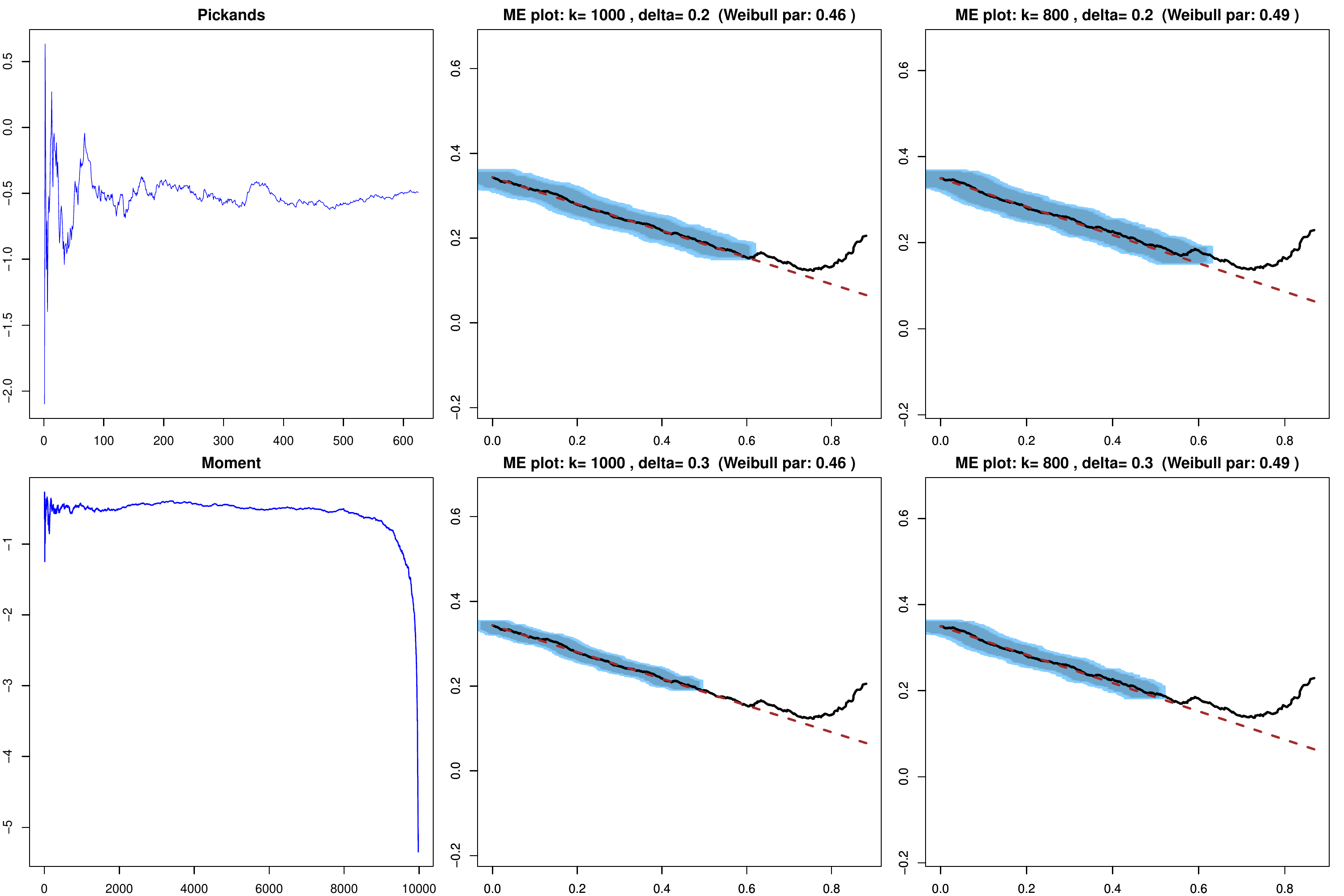}
 \caption{ME Plot for 10000 i.i.d. Generalized Pareto random variables with $\xi=-0.5$ and $\beta=1$ .}\label{fig:MEGPDneg}
 \end{figure}

Given an iid sample $X_1, \ldots, X_n \sim F$, we are concerned with detecting if $F\in D(G_{\xi})$ and if so whether $\xi$ is positive (the Fr\'echet case), zero (the Gumbel case) or
negative (the Weibull case).  The Fr\'echet case has been discussed in \cite{das:ghosh:2013} with examples. Hence we concern ourselves with the other two cases for the simulated examples. First we see how our confidence intervals work in simulated examples, and then use them on real data. In all the plots below, the light blue shade creates a $95\%$ confidence interval and the dark blue shade creates  $90\%$ confidence interval.

\begin{figure}[H]
\includegraphics[width=12cm]{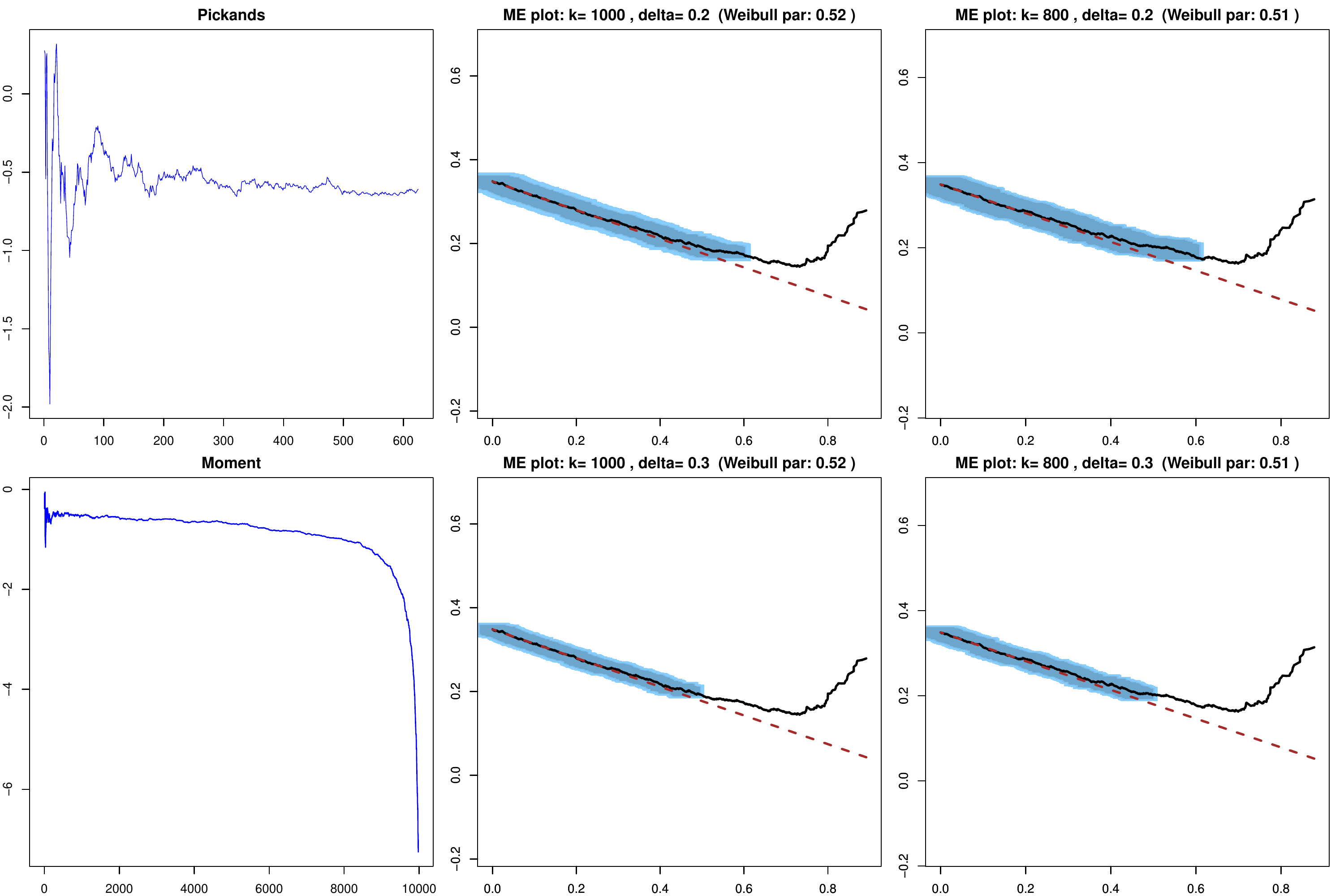}
\caption{ME Plot for 10000 i.i.d. $Beta(\alpha,\beta)$ random variables with $\alpha=2,\beta=2$.}\label{fig:MEbeta}
\end{figure}

\subsection{Simulated examples: Weibull} 
In this case we have $F\in D(G_{\xi})$ with $\xi < 0$.
We consider two families of distributions here.

 \begin{enumerate}
 \item Consider  $F\sim GPD(\xi,\beta)$ with pdf given by
 $$F(x)= 1 - \left(1+ \xi \frac{x}{\beta}\right)^{-1/\xi}, ~~ 1+  \frac{\xi x}{\beta}>0, ~\beta>0,\xi<0.$$
Of course, $F\in D(G_{\xi})$ here and $F$ is in the Weibull domain of attraction if $\xi<0$. In fact the Uniform $(0,1)$ falls into this class with $\xi=-1$ and $\beta=1$.\\ For our simulation example we take  $\xi=-0.5, \beta=1$ and generate 10000 iid samples from the distribution. 
 The two plots in the left of Figure \ref{fig:MEGPDneg} are Pickands and Moment estimate of $\xi$ for increasing values of top order statistics used. They seem reasonable close to $-0.5$. For $k=800,1000$ and $\delta=0.2,0.3$ we create confidence bounds around the ME plot (in black) which clearly covers the dashed red line with slope $-0.5$.

 \item Next consider $F \sim \text{Beta} (a, b)$ with pdf given by
 $$f(x) = \frac{\Gamma(a+b)}{\Gamma(a)\Gamma(b)}x^{a-1}(1-x)^{b-1}, ~~ 0<x<1, ~~a,b>0.$$
 In this case $F\in D(G_{-1/b})$. We take the example $a=2, b=2$ where $F\in D(G_{-0.5})$. As observed in the previous example we see that the Pickands and Moment estimates approximate $-0.5$ well; see \ref{fig:MEbeta}. We again create confidence bounds with $k=800,1000$ and $\delta=0.2,0.3$ and observe that the bounds cover the dashed red line with slope $-0.5$.
 \end{enumerate}

Thus the detection in the Weibull family looks reasonable.

\subsection{Simulated examples: Gumbel}

Distributions in the Gumbel domain of attraction are harder to detect since a data sample has to form a plot with slope zero in this case, which is statistically unlikely. Hence confidence bounds help to an extent, although as we will see through the three examples below that, in practice, a plotting technique is helpful to different degrees in different cases.

\begin{figure}[H]
\includegraphics[width=12cm]{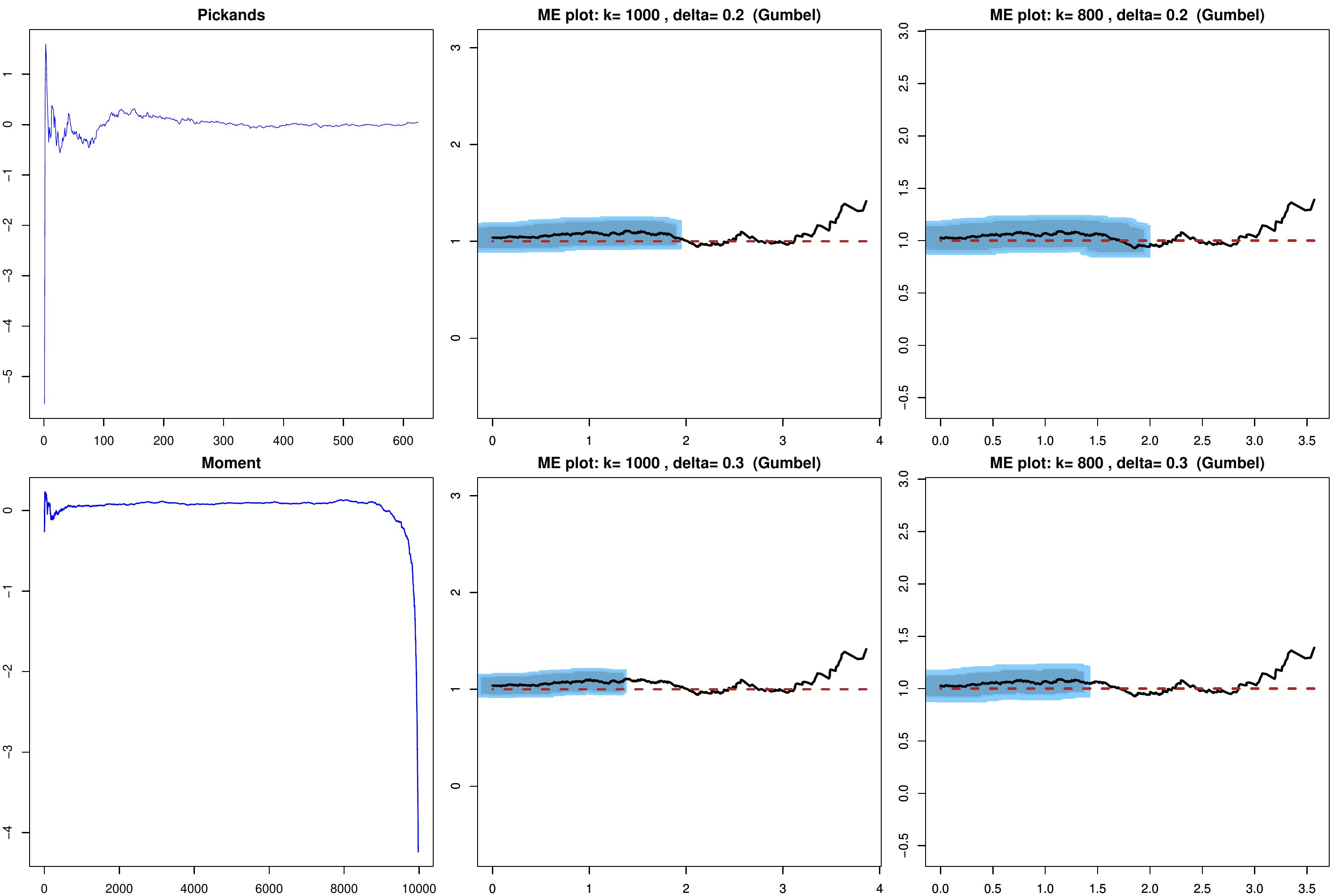}
\caption{ME Plot for 10000 i.i.d. Exponential random variables with $ \lambda=1$.}\label{fig:MEexp}
\end{figure}
\begin{figure}[H]
\includegraphics[width=12cm]{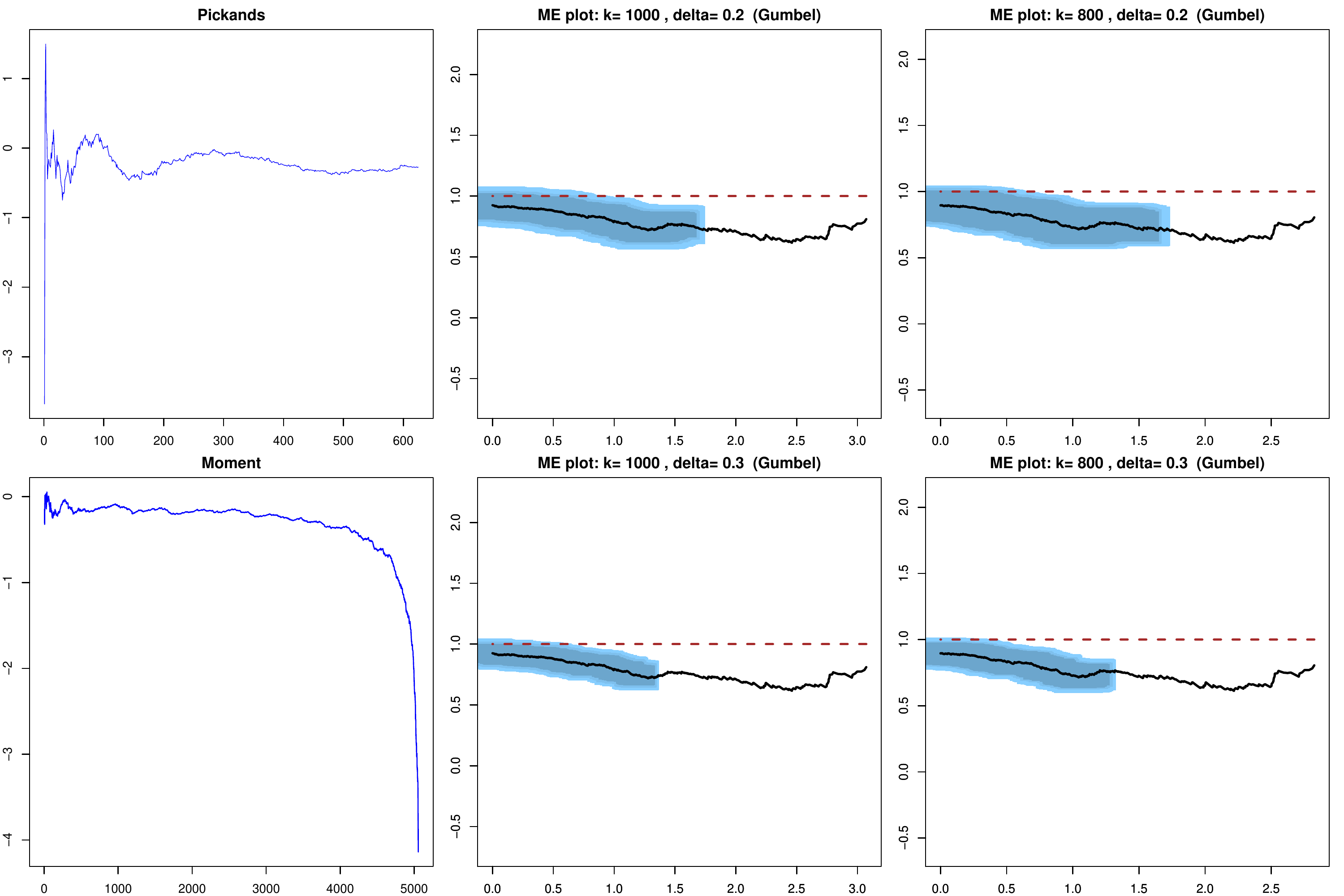}
\caption{ME Plot for 10000 i.i.d. Standard Normal random variables.}\label{fig:MEnorm}
\end{figure}
 \begin{enumerate}
 \item The first example is where $F$ follows $Exp(1)$. We generate 10000 iid samples from the distribution and create ME plots with parameters $k=800,1000$ and $\delta=0.2,0.3$; see Figure \ref{fig:MEexp}. The Pickands and Moment estimates are close to zero and the confidence intervals around the ME plot in the four different cases all cover the line with slope $0$ (and intercept $1$) as expected. So here the detection techniques works well.

\item The next data set we look at is a sample generated from $F$ which follows $N(0,1)$. We again generate 10000 iid samples from the distribution and make  ME plots with parameters $k=800,1000$ and $\delta=0.2,0.3$; see Figure \ref{fig:MEnorm}. In this case the Pickands estimate is close to zero but the  Moment estimate though close zero seem to be an underestimate. The confidence intervals around the ME plot in the four different cases all cover the dashed red line with slope $0$ (and intercept $1$) up to some point and then it doesn't. We can believe that $F\in D(G_0)$ but the case becomes less convincing than the previous example.

\begin{figure}[H]
\includegraphics[width=12cm]{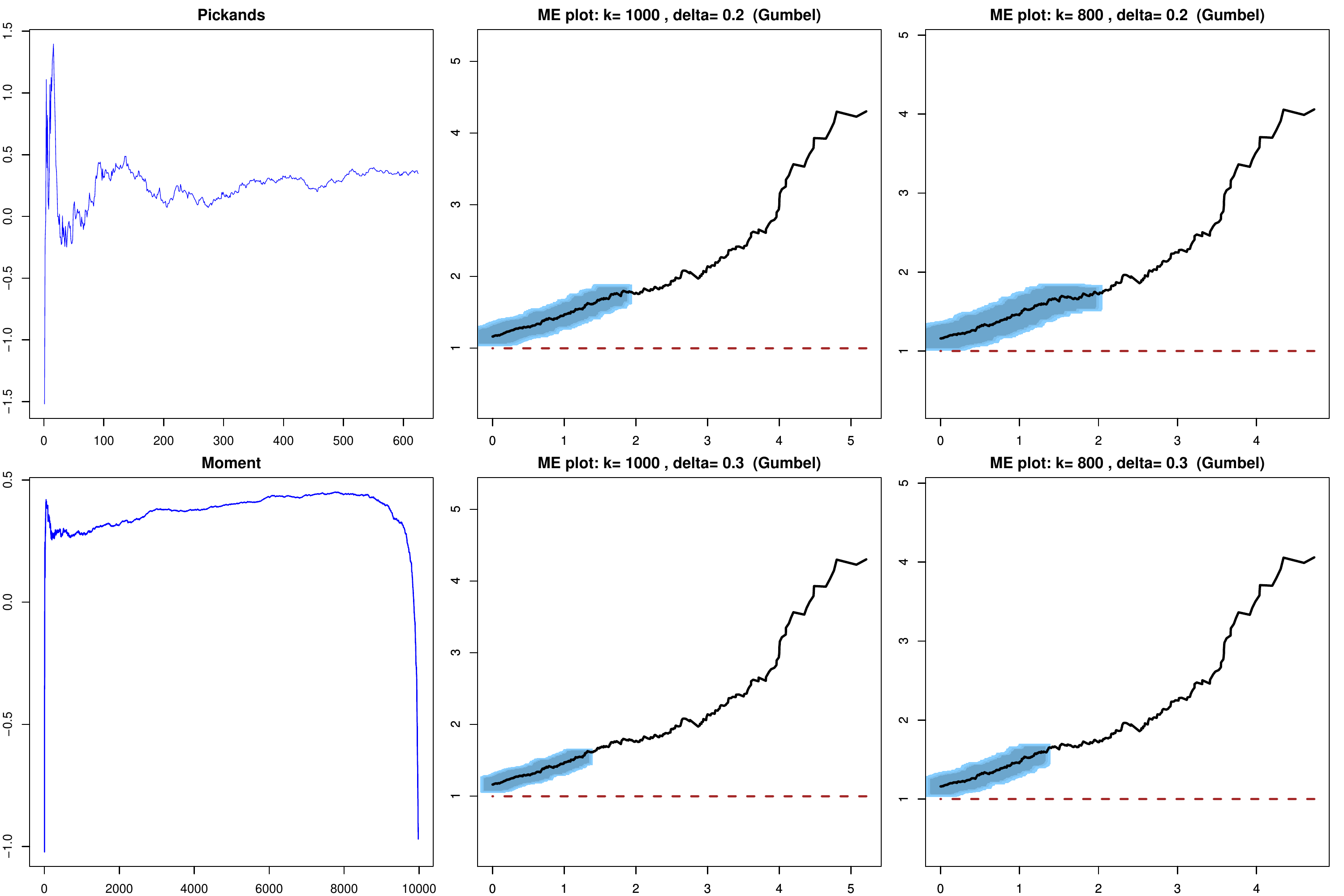}
\caption{ME Plot for 10000 i.i.d. Standard log-normal random variables.}\label{fig:MElognorm}
\end{figure}
\item Finally we look into $F$ which is a standard Lognormal distribution. It is known the a Lognormal distribution belongs to $D(G_0)$, but on the other hand we know that it has no finite moments (unlike the Normal or Exponential case). So it is on the one-hand sub-exponential or heavy-tailed although belong to a Gumbel domain of attraction.

We simulate 10000 iid samples from a standard Lognomal distribution and create ME plots as in the previous cases. The results are in Figure \ref{fig:MElognorm}. Both the Pickands estimate and the moment estimate of the extreme value parameter are much higher than the true value, that is zero. The ME plot with confidence intervals around it miss the target red dashed line of slope zero (and intercept $1$); a larger choice of $\delta$ would make the confidence intervals large enough to cover the line, but clearly our technique doesn't seem to perform so well here. 
 Since the Lognormal distribution has heavy tails we tend to have a positive slope of the ME plot as would happen in case when $F$ is in the Fr\'echet domain of attraction. Hence overall for detecting a Gumbel domain of attraction family we need to be more careful with this technique. 
 \end{enumerate}

\subsection{Observed data: Ozone concentration at Zurich urban area}
  It is of interest for environmental scientists to study \emph{ozone} concentration near urban conglomerations, as its presence in the atmosphere
  implies health risks related to respiratory diseases. Directive 2008/50/EC of the European Parliament puts the target value of ozone for its member states
  to be within $120 \mu g/m^3$. The directive says that as of January 1, 2010 ozone concentrations should not exceed this limit for more than 25 days in a calendar year
  where the daily calculation is based on maximum of daily 8-hour averages.
\begin{figure}[H]
\includegraphics[width=16cm]{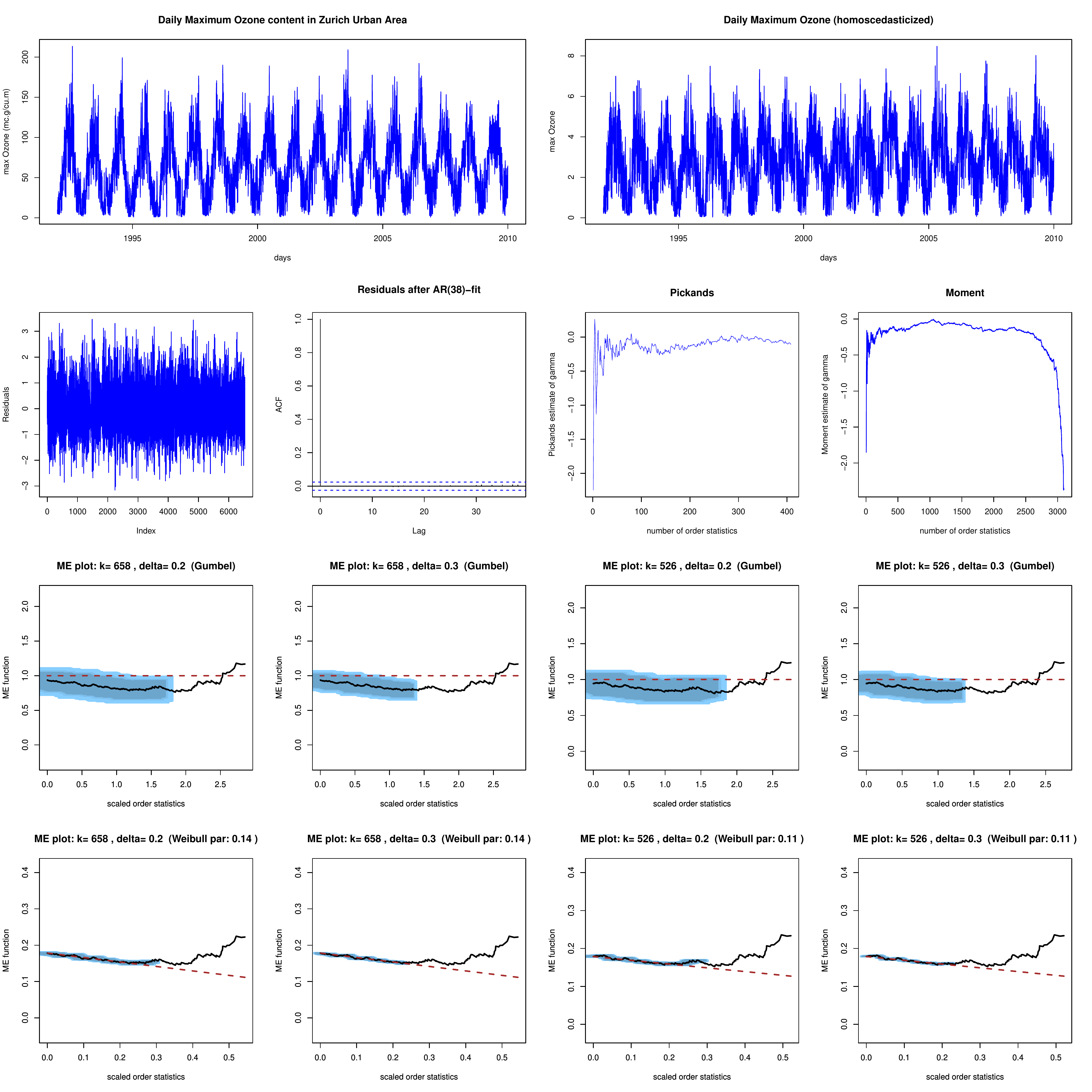}

\caption{ME Plot for Ozone concentration in Zurich urban area.}\label{fig:Zurich}
\end{figure}

  We study a data set, freely available from \url{www.eea.europa.eu}. The data set contains daily maxima of ozone concentration (in $\mu g/m^3$) from one station in Zurich, Switzerland (station code CH 0010A, Zurich-Kaserne) located $410 ~mts$ above sea-level. Data is observed from January 1, 1992 to December 31, 2009. Measurements
  were unavailable for 22 days, which we impute by the average value of ozone concentration on the same day for other available years. \\As seen in the top left plot in Figure
  \ref{fig:Zurich} the data clearly admits periodicity. {Moreover it is likely that the data is heteroscedastic. So we homoscedasticize the  data by dividing the value on each date by the standard deviation of the values on the same day over all the 18 years of data available. Since our techniques work for stationary data sets, we  fit an AR (38) process  to the data set (AR(38) is chosen by an AIC criterion) and observe (from the ACF; see second plot from the second line in Figure \ref{fig:Zurich}) that the residuals  (first plot in the second line) look independent.  Now we analyze the extremal behavior of the
residuals of the model. The Pickands and Moments estimates give a negative value but close to 0 and we can hypothesise that the sample is from a Weibull domain of attraction family. But, since the value of the parameter is close to 0 we also check whether the data is possibly from a Gumbel domain of attraction family. The confidence bounds (90\% deep blue and 95\% light blue) are created assuming $F\in D(G_0) $for $k=329$ and $ k=658$; which are 5\% and 10\% of the data set and with$\delta=0.2, 0.3$. Observe that the 90\% bounds tend to reject the hypothesis of the underlying $F\in D(G_0)$ and the 95\% do not. This is most likely a result of the parameter being close to zero.  

 On the other hand using the Pickands estimate to estimate the tail index we get $\xi=-0.14$ (for $k=658$) and 
$\xi=-0.11$ (for $k=329$) and the confidence bounds (again 90\% deep blue and 95\% light blue) for $\delta=0.2,0.3$ covers the straight line with the slopes $\xi$ quite well. Hence we are expect that the underlying distribution is in fact in a Weibull domain off attraction with parameter close to $\xi=0.1$.

\subsection{Observed data: Flow-rates at river Aare}

The other data we analyze is maximum daily flow-rate at river Aare. River Aare flows through Switzerland and some manufacturing and power plants located near the river are often concerned about flooding on the river. The data we analyse  has been collected from the Federal office of the Environment (FOEN), Bern and generously provided to us by Kernkraftwerk G\"osgen-D\"aniken.
It pertains  to daily maximum flow-rates of Aare at the measurement station Aare-Murgenthal (2063) measured in $m^3/sec$ from 1st January 1974 to 20th October, 2010. See also \url{www.hydrodaten .admin.ch/d/2063.htm}.

Note that the data admits to possibilities of measurement error since automated measurement at the specific station started only in 1993. Moreover, the control authorities aim to maintain the flow-rate of Aare at the Aare-Murgenthal (2063) station below $850 ~ m^3/sec$ and would do so by using opening or closing log-gates. This manually hinders the possibility of the data set being tuitionary. We were informed that such manual intervention has been done a couple of times.

To analyse the data, we first note the seasonality pattern in the data set; see top left plot in Figure \ref{fig:Aare}. Hence as in the previous example we fit an AR process and work with the residuals obtained after the model-fitting. Observe that the Pickands and Moment estimates both indicate towards a small but positive value of the extreme value parameter; but does not completely reject the possibility of it being zero. We again create $90\%$ (dark blue) and $95\%$ (light blue) confidence bounds under the Gumbel assumption for $k=673$ and $k=1345 $ (again $5\%$ and $10\%$) of the sample size
 and $\delta=0.1, 0.2$. The detection technique seems to reject that the underlying distribution $F\in D(G_0)$.
 
 Now we allow a Pickands estimate to chose the extreme value parameter which gives a value of around $\xi=0.16$ (for $k=673$ and $k=1345 $) and the confidence bounds seem to support that the data is from a distribution in the Fr\'echet domain of attraction.
 Thus we may conclude that flow-rate data at Murgenthal station is perhaps slightly heavy-tailed even if marginally so.

 \begin{figure}[H]
\includegraphics[width=16cm]{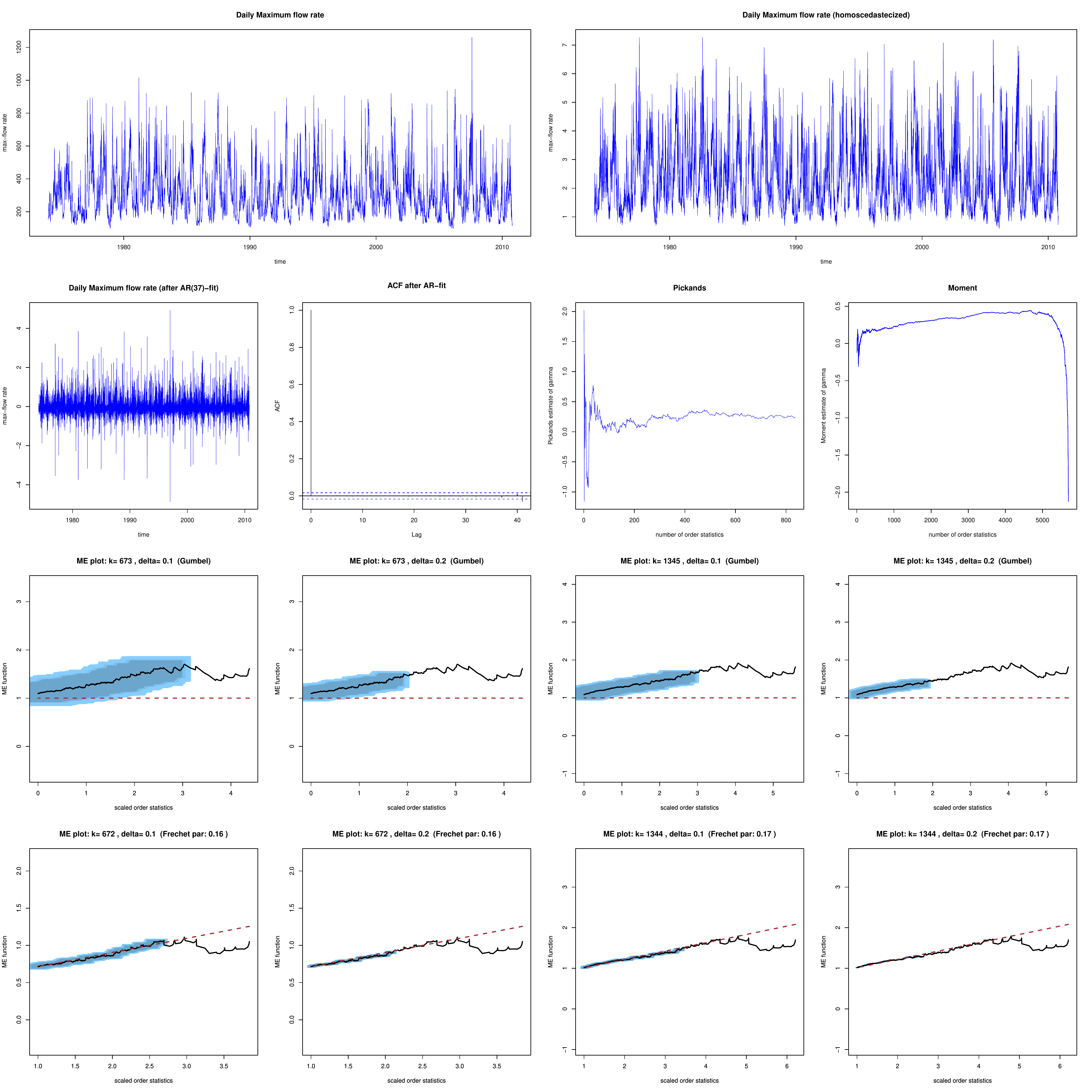}
\caption{ME Plot for Aare river flow data.}\label{fig:Aare}
\end{figure}


\bibliographystyle{imsart-nameyear}
\bibliography{bibfilenew}
\end{document}